\documentclass[11pt,reqno]{amsart}
\usepackage{amsmath,amssymb,url}

\usepackage{tikz}
\usetikzlibrary{calc}
\usetikzlibrary{plotmarks}

\newtheorem{theorem}{Theorem}[section]
\newtheorem{lemma}[theorem]{Lemma}
\newtheorem{corollary}[theorem]{Corollary}

\numberwithin{equation}{section}

\newcommand{\e}{\varepsilon}
\newcommand{\C}{{\mathbb C}}
\newcommand{\R}{{\mathbb R}}
\newcommand{\Real}{\operatorname{Re}}
\newcommand{\tr}{\operatorname{tr}}

\begin{document}

\title[Sums of magnetic eigenvalues]{Sums of magnetic eigenvalues are maximal on rotationally symmetric domains}

\author[Laugesen, Liang and Roy]{Richard S. Laugesen, Jian Liang and Arindam Roy}

\address{Department of Mathematics, University of Illinois, Urbana,
IL 61801, U.S.A.} \email{Laugesen\@@illinois.edu,Liang41\@@illinois.edu,Roy22\@@illinois.edu}
\date{\today}


\begin{abstract}
The sum of the first $n \geq 1$ energy levels of the planar Laplacian with constant magnetic field of given total flux is shown to be maximal among triangles for the equilateral triangle, under normalization of the ratio $(\text{moment of inertia})/(\text{area})^3$ on the domain. The result holds for both Dirichlet and Neumann boundary conditions, with an analogue for Robin (or de Gennes) boundary conditions too.

The square similarly maximizes the eigenvalue sum among parallelograms, and the disk maximizes among ellipses. More generally, a domain with rotational symmetry will maximize the magnetic eigenvalue sum among all linear images of that domain. 

These results are new even for the ground state energy ($n=1$).
\end{abstract}

\maketitle

\vspace*{-12pt}

\section{\bf Introduction}

Eigenvalues of the Laplacian on a plane domain represent energy levels of a quantum particle in two dimensions. These eigenvalues are known in closed form only for special types of domain, such as disks and rectangles. Consequently a great deal of effort has gone into proving upper and lower bounds on eigenvalues in terms of geometric properties of the domain, such as area and perimeter. 

For example, the Rayleigh--Faber--Krahn inequality says that the ground state energy $\lambda_1$ of a quantum particle with Planck constant $\hbar$ confined to a region of area $A$ is bounded below according to 
\[
\lambda_1 A \geq \hbar^2 j_{0,1}^2 \pi
\]
where $j_{0,1}$ is the first zero of the Bessel function $J_0$.  Equality holds for the disk. Such bounds provide not only hard estimates, but also soft intuition, for they indicate how geometric attributes of the domain constrain the analytic information encoded in the eigenvalues. 

A magnetic field imposed transversely through the domain makes the energy levels even more difficult to determine theoretically. In this paper we aim to discover geometrically sharp estimates on such magnetic eigenvalues. 

Given a bounded plane domain $\Omega$, write $\lambda_j \big( \Omega , \hbar, \beta \big)$ for the $j$th eigenvalue of the Dirichlet Laplacian with Planck constant $\hbar>0$, under a constant transverse magnetic field $(0,0,\beta)$.  These Dirichlet eigenvalues are defined precisely in the next section, as are the Neumann eigenvalues $\mu_j$ (assuming the domain has Lipschitz boundary). The eigenvalue equations are
\begin{align}
\text{Dirichlet:} \qquad & 
\begin{cases}
(i\hbar \nabla + F)^2 u = \lambda u \quad \text{in $\Omega$} \\
u = 0 \quad \text{on $\partial \Omega$}
\end{cases} \label{eqbc1}
\\
\text{Neumann:} \qquad & 
\begin{cases}
(i\hbar \nabla + F)^2 u = \mu u \quad \text{in $\Omega$} \\
\vec{n}  \cdot (\hbar\nabla-iF)u = 0 \quad \text{on $\partial \Omega$}
\end{cases} \label{eqbc2}
\end{align}
where the vector potential $F(x)=\frac{\beta}{2}(-x_2,x_1)$ creates a field \text{$\nabla \times F = (0,0,\beta)$}. 

Let $A$ be the area and $I$ the moment of inertia of $\Omega$ about its centroid:
\[
  I = \int_\Omega |x-c|^2 \, dx
\]
with centroid $c = \frac{1}{A} \int_\Omega x \, dx$.

Our method works by linearly transforming a bounded plane domain $D$ having $N$-fold rotational symmetry, such as in Figure~\ref{fig:lineartrans}. The main result, Theorem~\ref{mag}, says that sums of magnetic eigenvalues decrease under linear transformation of the rotationally symmetric domain, provided the Planck constant and field strength are transformed suitably too. This theorem encompasses both Dirichlet and Neumann boundary conditions, and Robin--de Gennes conditions are treated in Theorem~\ref{magrobin}. The ability to handle all three boundary conditions is a striking feature of our method.

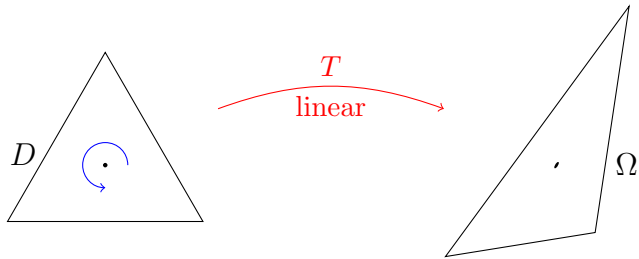
\begin{figure}[t]
  \begin{center}
    \begin{tikzpicture}[scale=1.5]
      \draw +(90:1) -- +(210:1)node [pos=0.6,left] {\large $D$}  --  +(330:1) -- cycle;
      \fill circle (0.02);
      \draw[blue,->] (0:0.2) arc (0:270:0.2);
      \draw[cm={1,1,0,1,(4,0)}] [rotate=60] +(110:1) -- +(230:1) -- +(350:1) node [pos=0.3,right] {\large $\Omega$} -- cycle;
      \fill [cm={1,1,0,1,(4,0)}] circle (0.02);
      \draw[red,->] (1,0.5) [in=160,out=20] to (3,0.5) node [pos=0.5,above] {$T$}
      node [pos=0.5,below] {linear};
    \end{tikzpicture}
  \end{center}
  \caption{A plane domain with $3$-fold rotational symmetry, and its image under a linear map $T$.}
  \label{fig:lineartrans}
\end{figure}

By expressing our main theorem geometrically, we will obtain:
\begin{corollary}
\label{2dim}
Assume $D$ is a bounded plane domain with rotational symmetry of order greater than or equal to $3$, and that $\Omega$ is the image of $D$ under a linear transformation. Fix $\beta \in \R$. 

Then for each $n \geq 1$, the normalized eigenvalue sum
\begin{equation} \label{flux_intro}
\Big[ \lambda_1 \big( \Omega , \hbar, \frac{\beta}{A} \big) + \dots + \lambda_n \big( \Omega , \hbar, \frac{\beta}{A} \big) \Big] \frac{A^3}{I}
\end{equation}
is maximal when $\Omega=D$, for each $n \geq 1$. 

Maximality holds also for dilated, rotated, and reflected images of $D$; and when $n=1$, every maximizing $\Omega$ is of that type.

The same conclusions hold for sums of Neumann eigenvalues.
\end{corollary}
The corollary is proved in Section~\ref{2dim_proof}. 

Notice the total flux of the magnetic field is the same for each domain, in the corollary, since multiplying the field strength $\beta/A$ by area $A$ gives $\beta$.

In particular, the corollary implies that the normalized eigenvalue sum \eqref{flux_intro} is maximal among triangles for the equilateral, maximal among parallelograms for the square, and maximal among ellipses for the disk. To apply this result in practice, one would like explicit formulas for the eigenvalues of the equilateral triangle and the square --- but we do not know whether such formulas exist for $\beta \neq 0$.

Incidentally, for triangles the moment of inertia can be calculated in terms of the side lengths as $I = (l_1^2+l_2^2+l_3^2)A/36$, while for a parallelogram with adjacent side lengths $l_1,l_2$, the moment of inertia equals $I = (l_1^2+l_2^2)A/12$. 

Two reasons for studying eigenvalue sums are that the sum represents the energy needed to fill the lowest $n$ states under the Pauli exclusion principle, and that summability methods improve the behavior of high eigenvalues, which are difficult to study directly. An example of the improvement provided by summation is that although the P\'olya conjecture remains open, claiming the Weyl asymptotic provides a lower bound for each eigenvalue of the Dirichlet Laplacian, Li and Yau  \cite{LY83}  were able to show that \emph{sums} of these eigenvalues are indeed bounded below by the analogous Weyl asymptotic.

\subsection*{Prior work}
The topic originates with an old result of P\'{o}lya on the  fundamental tone of a membrane (as stated in \cite{P52} and proved in \cite[Chapter IV]{PS53}). He established the case $n=1$ of this paper, for vanishing magnetic field; that is, he studied the first eigenvalue of the usual Dirichlet Laplacian. His method relied heavily on uniqueness of the fundamental mode, and thus it failed to extend to Neumann boundary conditions, higher eigenvalue sums or the magnetic spectrum, all of which we treat in this paper.

Our work relies on the method of Rotations and Tight Frames, developed recently by Laugesen and Siudeja \cite{LS11b} for finding upper bounds on eigenvalue sums of the Laplacian. We extend this technique to handle magnetic fields. The resulting theorems are new even when $n=1$. 

The magnetic situation presents two new challenges. First, the interactions of the momentum (or gradient) operator with the magnetic field must be averaged over the $N$-fold rotations of $D$; see the quantities $Q_2$ and $Q_3$ in Section~\ref{mag_proof}. Second, a new proof must be found for the equality statement (see Section~\ref{equality_proof}), because the eigenfunction $u$ is complex-valued in the magnetic situation, and so $u$ and $\overline{u}$ are non-equal in formulas \eqref{eigen1new} and \eqref{eigen2new}.

The only prior geometrically sharp estimate we know, for magnetic eigenvalues, is the Faber--Krahn type lower bound of L. Erd\"{o}s \cite{E96}. That result says the normalized ground state energy $\lambda_1(\Omega,\hbar,\frac{\beta}{A}) A$ is minimal for the disk among all two dimensional domains. Notice the total flux is once again the same, for each domain considered. 

The lower bound of Erd\"{o}s relates to our upper bound in Corollary~\ref{2dim} as follows. One can rewrite our functional in \eqref{flux_intro} as $( \lambda_1 + \dots + \lambda_n )A \cdot  A^2/I$. Thus it arises from multiplying a Faber--Krahn term $( \lambda_1 + \dots + \lambda_n )A$ that is normalized by area with a purely geometric, scale-invariant term $A^2/I$ that penalizes long, thin domains. Erd\"{o}s studied solely the Faber--Krahn term, and only for $n=1$, although his work does apply to arbitrary domains. We would like to extend our work to arbitrary domains too, but the task seems challenging even in the absence of magnetic field \cite[{\S}4]{LS11b}.

The scarcity of isoperimetric type inequalities for magnetic eigenvalues contrasts with a profusion of results in the nonmagnetic setting, where one may consult the surveys by Ashbaugh and Benguria \cite{AB07} or Benguria and Linde \cite{BL08}, and the monographs of Bandle \cite{B80}, Henrot \cite{He06}, Kawohl \cite{K85}, Kesavan \cite{K06} and P\'{o}lya--Szeg\H{o} \cite{PS51}. In particular, triangular domains have been much studied \cite{AF11,F07,LPS11,LS09,LS11a,LS11b,LR10,S10}. Our aim in this paper is to begin developing a magnetic isoperimetric theory of comparable richness.

\subsection*{Asymptotically sharp inequalities (semi-classical constants)}
Our work in this paper is ``geometrically sharp'', since an extremal domain exists for each fixed $n$. The Li--Yau inequality mentioned above (for eigenvalue sums in the absence of magnetic field) is not geometrically sharp, but is asymptotically sharp since equality holds for each domain in the limit as $n \to \infty$, by the Weyl asymptotic. Other asymptotically sharp bounds on eigenvalue sum functionals were studied recently by Dolbeault, Geisinger, Harrell, Hermi, Kr\"oger, Laptev, Loss, and Weidl (see \cite{DLL08,GLW10,HH08,Kr94} and references therein).

Fewer such estimates are known with magnetic field, and some natural conjectures turn out to be false. Notably, the magnetic P\'{o}lya conjecture has been disproved by Frank, Loss and Weidl \cite{FLW09}, who constructed counterexamples even for square domains. A magnetic Li--Yau inequality nonetheless holds true for eigenvalue sums, by work of Erd\"{o}s, Loss and Vougalter \cite{ELV00} extending results of Laptev and Weidl \cite{LW00}. 

\subsection*{Higher dimensions} We hope to extend our inequalities to higher dimensions in a later paper. The results will be more complicated, for two reasons. First, the moment of inertia must be evaluated on an ``inverse'' domain, as seen already with vanishing magnetic field \cite{LS11c}. Second, the field is no longer perpendicular to the domain.

No sharp Faber--Krahn lower bound is known for the ground state energy of the magnetic Laplacian, in higher dimensions. The lower bound of Erd\"{o}s in the plane remains to be extended.

\section{\bf Assumptions and notation}
\label{notation}

\subsection*{Vector potentials} Throughout the paper we consider a constant magnetic field on the plane, written
\[
B=(0,0,\beta)
\]
for some fixed $\beta \in \R$. The field can be generated from the curl of a suitable vector potential $F \in C^\infty(\R^2;\R^2)$, as
\[
B = \nabla \times F = (0,0,\partial_1 F_2 - \partial_2 F_1).
\]
Most commonly we employ the potential
\[
F(x) = \frac{\beta}{2} (-x_2,x_1) ,
\]
but other choices are allowed too, such as $\beta(-x_2,0)$ or $\beta(0,x_1)$. This non-uniqueness of the potential illustrates the principle of \emph{gauge invariance}, whereby the magnetic field is unchanged by adding a gradient vector to the potential, because the curl of a gradient equals zero. 

\subsection*{Eigenvalues}
Consider a bounded plane domain $\Omega$. Denote by $\lambda_j(\Omega,\hbar,\beta)$ the \emph{Dirichlet} eigenvalues of the Laplacian on $\Omega $ with Planck constant $\hbar>0$ and constant magnetic field $B=(0,0,\beta)$. These eigenvalues form an increasing sequence 
\[
0 < \lambda_1 \leq \lambda_2 \leq \lambda_3 \leq \dots \to \infty
\]
and are determined from the Rayleigh quotient
\[
R[u]  = \frac{\int_ \Omega |(i\hbar \nabla+F) u|^2 \, dx}{\int_ \Omega |u|^2 \, dx} \qquad \text{where\ } u \in H^1_0(\Omega;\C) , \ u \not \equiv 0 ,
\]
by the usual minimax variational principles. In particular, $\lambda_1 = \min_u R[u]$. Note the gradient operator $\nabla$ and the vector potential $F$ are regarded as row vectors. 

The existence of these eigenvalues, and of a corresponding orthonormal basis of smooth eigenfunctions, follows from standard elliptic theory \cite[Corollary III.7.8]{S79}. The key facts are that $H^1_0$ imbeds compactly into $L^2$ and the Rayleigh quotient can be written $R[u]= Q[u,u]/ \langle u , u \rangle_{L^2(\Omega)} $ where the sesquilinear form 
\[
Q[u,v] = \big\langle (i\hbar \nabla+F) u , (i\hbar \nabla+F) v \big\rangle_{L^2(\Omega)}
\]
satisfies the following three properties. It is: (i) conjugate symmetric, (ii) continuous as a function of $u,v \in H^1_0(\Omega;\C)$, and (iii) $H^1$-coercive after adding a multiple of the $L^2$-inner product (meaning $Q[u,u] + c_1 \lVert u \rVert_{L^2}^2 \geq c_2 \lVert u \rVert_{H^1}^2$ for all $u \in H^1(\Omega;\C)$, for some $c_1,c_2>0$). 

The eigenfunction equation \eqref{eqbc1} follows from the Euler--Lagrange condition for a critical point of the Rayleigh quotient. Note the eigenfunctions depend on the choice of vector potential, but the eigenvalues do not: they depend merely on the field strength parameter $\beta$, by gauge invariance (see Lemma~\ref{indeppot}). Thus our notation $\lambda(\Omega,\hbar,\beta)$ need only indicate the dependence of the eigenvalue on the field parameter $\beta$, not the vector potential. 

The first eigenvalue is positive by diamagnetic comparison with the first eigenvalue of the Dirichlet Laplacian ($\beta=0$), or else by Lemma~\ref{poseigen}. (Also $\lambda_1(\Omega,\hbar,\beta) \geq \lambda_1(\R^2,\hbar,\beta) = \hbar |\beta|$ by domain monotonicity, but we will not need this fact.) 

The \emph{Neumann} eigenvalues $\mu_j(\Omega,\hbar,\beta)$ arise from the same Rayleigh quotient $R[u]$, but the class of trial functions is larger, namely $u \in H^1(\Omega;\C)$; we make the standing assumption that the domain has Lipschitz boundary so that $H^1$ imbeds compactly into $L^2$. The first Neumann eigenvalue $\mu_1$ is positive whenever $\beta \neq 0$, by Lemma \ref{poseigen}. Note the boundary conditions in \eqref{eqbc2} arise \emph{naturally} in the Neumann case, from the Euler--Lagrange condition.

The magnetic eigenvalues are invariant with respect to rigid motions and dilations of the domain. See Appendix~\ref{invariance} for precise statements.

\subsection*{Matrix notation}
Given a real or complex matrix $M$, write its Hilbert--Schmidt norm as
\[
\lVert M \rVert_{HS}=\big( \sum_{j,k} |M_{jk}|^2 \big)^{\! 1/2} = (\tr M M^\dagger)^{\! 1/2}
\]
where $M^\dagger$ denotes the complex conjugate of the transpose matrix. 

For example, the identity matrix has Hilbert--Schmidt norm equal to $\sqrt{2}$.

\section{\bf The main results: sharp upper bounds on eigenvalue sums}
\label{results}

Fix the Planck constant $\hbar>0$ and field parameter $\beta \in \R$, throughout this section.

\subsection*{Dirichlet and Neumann eigenvalues}
We start by linearly transforming a rotationally symmetric plane domain $D$. The Planck constant and field strength are transformed suitably too, in order to obtain a sharp result. 

\begin{theorem}
\label{mag}
If the bounded plane domain $D$ has rotational symmetry of order greater than or equal to $3$, then
\[
\sum\limits_{j=1}^{n} \lambda_j\Big({T(D)}, \frac{\sqrt{2}}{\lVert T^{-1} \rVert_{HS}} \hbar, \frac{\sqrt{2}}{\lVert T^{-1} \rVert_{HS}} \frac{\beta}{|\det T|} \Big) \leq \sum\limits_{j=1}^{n} \lambda_j\big(D,\hbar,\beta)
\]
for each $n \geq 1$ and each invertible linear transformation $T$ of $\R^2$.

Equality holds if $T$ is a scalar multiple of an orthogonal transformation. For nonzero magnetic fields ($\beta \neq 0$), equality holds for the first eigenvalue ($n=1$) if and only if $T$ is a scalar multiple of an orthogonal transformation.

The same results hold for Neumann eigenvalues.
\end{theorem}
The proof is in Sections~\ref{mag_proof} and \ref{equality_proof}. 

The rotationally symmetric domain $D$ in the theorem need not be convex, or a regular polygon, or have any axis of symmetry. For example, it could be shaped like a three-bladed propeller.

The equality statement for the first eigenvalue under \emph{zero} magnetic field is more complicated than in the theorem, because rectangles (which do not possess rotational symmetry) can also be extremal \cite[Theorem~3.1]{LS11b}. On the other hand, the equality case in this paper requires new ideas too, in order to handle the magnetic interaction terms.

\subsection*{Robin--de Gennes eigenvalues}

Analogous results can be established for the boundary condition of the third kind. This boundary condition was studied in thermodynamics by Robin and in superconductivity by de Gennes. First we need some definitions. Denote the \emph{Robin--de Gennes} eigenvalues by $\rho_j(\Omega,\hbar,\beta,\sigma)$ where the constant $\sigma>0$ is the Robin parameter. The Rayleigh quotient is
\begin{equation} \label{robinrayleigh}
R[u] = \frac{\int_ \Omega |(i \hbar \nabla+F) u|^2 \, dx + \sigma \int_{\partial \Omega} |u|^2 \, ds}{\int_ \Omega |u|^2 \, dx} \qquad \text{where\ } u \in H^1(\Omega;\C) , \ u \not \equiv 0 .
\end{equation}
As in the Neumann case, we assume the domain has Lipschitz boundary so that the spectrum is well defined. 

The eigenvalue equation and boundary conditions are easily deduced from the variational characterization of the eigenvalues. They are:
\[
\text{Robin-de Gennes:} \qquad 
\begin{cases}
(i\hbar \nabla + F)^2 u = \rho u \quad \text{in $\Omega $,} \\
\vec{n}  \cdot \hbar (\hbar\nabla-iF)u + \sigma u= 0 \quad \text{on $\partial \Omega $.}
\end{cases}
\]
Notice the Robin--de Gennes case reduces to Neumann when $\sigma=0$, and reduces formally to the Dirichlet case when $\sigma=\infty$. 

The first eigenvalue $\rho_1$ is positive for each value of $\beta$, by Lemma~\ref{poseigen}. 

The Robin--de Gennes eigenvalue sums are bounded sharply by the next theorem. 
\begin{theorem}
\label{magrobin}
If $D$ has rotational symmetry of order greater than or equal to $3$, then
\[
\sum\limits_{j=1}^{n} \rho_j\Big({T(D)}, \frac{\sqrt{2}}{\lVert T^{-1} \rVert_{HS}} \hbar, \frac{\sqrt{2}}{\lVert T^{-1} \rVert_{HS}} \frac{\beta}{|\det T|}, \frac{\sqrt{2}}{\lVert \, T^{-1} \rVert_{HS}} \sigma \Big) \leq \sum\limits_{j=1}^{n} \rho_j\big(D,\hbar,\beta,\sigma)
\]
for each $n \geq 1$ and each invertible linear transformation $T$ of $\R^2$.

Equality holds if $T$ is a scalar multiple of an orthogonal transformation. For nonzero magnetic fields ($\beta \neq 0$), equality holds for the first eigenvalue ($n=1$) if and only if $T$ is a scalar multiple of an orthogonal transformation.
\end{theorem}
The proof is in Section~\ref{magrobin_proof}. 

One can also express the theorem in geometric terms, similarly to Corollary~\ref{2dim}. (See the nonmagnetic version in \cite[Corollary~3.4]{LS11b}.)

\section{\bf Tight frame identities}
\label{tightframe_sec}

When proving the main theorem, we will need to average certain matrices with respect to conjugation by the rotation group of order $N$. Let the matrix $U_m$ represent rotation of the plane by angle $2\pi m/N$. Write $\mathrm{Id}$ for the identity matrix. 
\begin{lemma}
\label{tightframe}
If $N \geq 3$, then for every $2 \times 2$ real matrix $M$ one has
\[
\frac{1}{N} \sum_{m=1}^N U_m M U_m^\dagger = \big( \frac{1}{2}\tr M \big) \mathrm{Id} + \frac{1}{2} (M - M^\dagger).
\]
\end{lemma}
\begin{proof}
Begin by decomposing $M$ into its symmetric and antisymmetric parts as $M=M_s+M_a$, where $M_s=(M+M^\dagger)/2$ and $M_a=(M-M^\dagger)/2$. To handle the symmetric part we argue as for Schur's Lemma in representation theory: let 
\[
V=\frac{1}{N} \sum_{m=1}^N U_m M_s U_m^\dagger
\]
and observe that $V$ is symmetric. Notice $U_1 V U_1^\dagger=V$ for each $m$ (because $U_1 U_m = U_{m+1}$ and $U_{N+1}=U_1$). Thus $V$ commutes with the rotation $U_1$. Let $\alpha \in \R$ be an eigenvalue of the symmetric matrix $V$, with eigenvector $x$. Then
\[
  V(U_1 x)=U_1(Vx)=\alpha(U_1 x).
\]
Hence $U_1 x$ is also an eigenvector belonging to $\alpha$. Since $x$ and $U_1 x$ are linearly independent (noting that $U_1$ rotates by angle $\frac{2\pi}{N}<\pi$, because $N \geq 3$), by taking linear combinations we deduce that every vector in $\R^2$ is an eigenvector of $V$ with eigenvalue $\alpha$. Hence $V=\alpha \, \mathrm{Id}$.
Taking the trace yields $\alpha =  \frac{1}{2} \tr V$. Also 
\[
\tr V= \frac{1}{N} \sum_{m=1}^N \tr M_s = \tr M_s = \tr M .
\]
Therefore $V= \frac{1}{2} \big( \tr M \big) \mathrm{Id}$, which gives the first part of the formula in the Lemma.

Since $M$ is real, its antisymmetric part can be written $M_a = 
b \left( \begin{smallmatrix}
0 & -1 \\ 1 & 0
\end{smallmatrix} \right)$ for some $b \in \R$. Thus $M_a$ consists of rotation by $\pi/2$ followed by rescaling by $b$. Hence $M_a$ commutes with the rotation $U_m$, and so
\[
\frac{1}{N} \sum_{m=1}^N U_m M_a U_m^\dagger = M_a = \frac{1}{2} (M - M^\dagger) ,
\]
which proves the second term of the formula in the Lemma.
\end{proof}

Next we collect some immediate consequences of Lemma~\ref{tightframe}.
\begin{lemma} \label{consequences}
Let $N \geq 3$, and suppose $M$ and $T$ are $2 \times 2$ real matrices, with $T$ invertible. Then
\begin{align*}
\frac{1}{N} \sum_{m=1}^N U_m T^{-1} T^{-\dagger} U_m^\dagger & = \frac{1}{2} \lVert T^{-1} \rVert_{HS}^2 \, \mathrm{Id} \, , \\
\frac{1}{N} \sum_{m=1}^N U_m T^\dagger M^\dagger M T U_m^\dagger & = \frac{1}{2} \lVert MT \rVert_{HS}^2 \, \mathrm{Id} \, . \notag
\end{align*}
If $\tr M = 0$ then
\[
\frac{1}{N} \sum_{m=1}^N U_m T^{-1} M T U_m^\dagger = \frac{1}{2}  \big( T^{-1} M T - (T^{-1}MT)^\dagger \big).
\]
\end{lemma}

Also, by applying Lemma~\ref{tightframe} with $M=T^\dagger T$ and then putting $y^\dagger$ on the left and $y$ on the right, we find: 
\begin{lemma} \label{variant}
Let $N \geq 3$ and suppose $T$ is a $2 \times 2$ real matrix and $y \in \R^2$ is a unit column vector. Then 
\[
\frac{1}{N} \sum_{m=1}^N |T U_m^{-1} y|^2  = \frac{1}{2} \lVert T \rVert_{HS}^2 .
\]
\end{lemma}

\subsection*{Connection to equiangular tight frames} Fix a vector $y \in \R^2$ of length $1$. Then  $M=y y^\dagger$ is a symmetric $2 \times 2$ matrix. Lemma~\ref{tightframe} implies that 
\[
\frac{1}{N} \sum_{m=1}^N U_m y y^\dagger U_m^\dagger = \frac{1}{2} \mathrm{Id} .
\]
Conjugating with $x \in \R^2$ yields
\[
\frac{1}{N} \sum_{m=1}^N x^\dagger U_m y y^\dagger U_m^\dagger x = \frac{1}{2} |x|^2 ,
\]
or in other words
\[
\sum_{m=1}^N |x \cdot U_m y|^2  = \frac{N}{2} |x|^2 , \qquad x \in \R^2 .
\]
This Plancherel type identity says that the system of vectors $\{ U_m y \}_{m=1}^N$ (consisting of the $N$th roots of unity, rotated to start at $y$) forms a \emph{tight frame} with constant $N/2$. 

More information on frames and their applications in Hilbert spaces may be found in the monograph by Christensen \cite{Ch03} or the text by Han \emph{et al.}\ \cite{H07}.

\section{\bf Proof of Theorem~\ref{mag}: the inequality}
\label{mag_proof}

We prove the Dirichlet inequality in the theorem. The Neumann proof is virtually identical.

Without loss of generality, we may assume $T$ is diagonal with 
\[
T = \begin{pmatrix} t_1 & 0 \\ 0 & t_2 \end{pmatrix}
\]
where $t_1, t_2 > 0$, by a singular value decomposition of $T$ and using invariance of the spectrum under rotations and reflections (Appendix~\ref{invariance}).

Choose the vector potential to be 
\[
F(x) = \frac{\beta}{2} (-x_2,x_1) = \beta (Mx)^\dagger \qquad  \text{where} \qquad 
M = 
\frac{1}{2} \begin{pmatrix}
0 & -1 \\ 1 & 0
\end{pmatrix} ,
\]
so that the magnetic field is $\nabla \times F = (0,0,\beta)$ as required. 
(\emph{Aside.} This vector potential points tangentially to circles centered at the origin, and so it interacts well with rotations of the domain.) 

For brevity in what follows, we write
\[
b =  \frac{\sqrt{2}}{\lVert T^{-1} \rVert_{HS}} = \sqrt{\frac{2}{t_1^{-2}+t_2^{-2}}}, \qquad c = \frac{\sqrt{2}}{\lVert T^{-1} \rVert_{HS}} \frac{1}{|\det T|} = \sqrt{\frac{2}{t_1^2+t_2^2}}.
\]
Then $bc=2 \big/ \big( t_1 t_2^{-1} + t_1^{-1} t_2 \big)$.

We will need the Rayleigh--Poincar\'{e} variational principle \cite[p.\,98]{B80}, which characterizes the sum of the first $n \geq 1$ eigenvalues as
\begin{align*}
& \sum\limits_{j=1}^n \lambda_j \\
& = \min \big\{ \sum\limits_{j=1}^n R[v_j] : v_1, \dots,v_n \in H^1_0(\Omega;\C)\text{\ are pairwise $L^2$-orthogonal} \big\} .
\end{align*}
Specifically, we will construct trial functions $v_j$ on the domain $T(D)$ by linearly transplanting eigenfunctions $u_j$ of $D$, and then we will average with respect to the rotations of $D$ by means of the tight frame lemmas in Section~\ref{tightframe_sec}. 

Let $u_1,u_2,u_3,\ldots$ be orthonormal eigenfunctions on $D$ for the vector potential $F$, corresponding to the eigenvalues $\lambda_j(D,\hbar,\beta), j=1,2,3,\ldots$. Consider an orthogonal $2 \times 2$ matrix $U$ that fixes $D$, so that $U(D)=D$. Define trial functions
\[
v_j = u_j \circ U \circ T^{-1}
\]
on the domain $E=T(D)$, noting $v_j \in H^1_0(E;\C)$ because $u_j \in H^1_0(D;\C)$.
The functions $v_j$ are pairwise orthogonal, since
\[
\int_E v_j \overline{v_k} \, dx = \int_D u_j \overline{u_k} \, dx \cdot |\det TU^{-1}| = 0
\]
when $j \neq k$. Thus by the Rayleigh--Poincar\'{e} principle, we have
\begin{equation}
\label{rayleighprinc}
\sum_{j=1}^n \lambda_j(E,b\hbar,c\beta) \leq \sum_{j=1}^n \frac{\int_E |(i b \hbar\nabla+c F) v_j|^2 \, dx}{\int_E |v_j|^2 \, dx} .
\end{equation}
For each function $v=v_j$, we change variable in the right side of \eqref{rayleighprinc} with $x \mapsto TU^{-1}x$ and use that the $u_j$ are normalized in $L^2$, thus finding 
\begin{align*}
&  \frac{\int_E |(i b\hbar\nabla+ c F) v|^2 \, dx}{\int_E |v|^2 \, dx} \\
& = \frac{\int_D \big| i b \hbar\nabla u(x)UT^{-1}+ c \beta(MTU^{-1}x)^\dagger u(x) \big|^2 \, dx \cdot |\det TU^{-1}|}{\int_D |u|^2 \, dx \cdot |\det TU^{-1}|} \\
& = \int_D \big| i b \hbar\nabla u(x)UT^{-1}+c \beta(MTU^{-1}x)^\dagger u(x) \big|^2 \, dx \\
& = \int_D \big( Q_1 + Q_2 + Q_3 \big) \, dx
\end{align*}
by expanding the square, where 
\begin{align*}
Q_1 & = b^2 \hbar^2 \nabla u \big[ UT^{-1} T^{-\dagger} U^\dagger \big] (\nabla u)^\dagger , \\
Q_2 & = 2 b c \hbar \beta \Real \big\{ i \overline{u} \nabla u \big[ UT^{-1}MTU^\dagger \big] x \big\} , \\
Q_3 & = c^2 \beta^2 |u|^2 x^\dagger \big[ U T^\dagger M^\dagger M T U^\dagger \big] x .
\end{align*}

Let $N \geq 3$ be the order of rotational symmetry of $D$, and choose $U$ to be the matrix $U_m$ representing rotation by angle $2\pi m/N$. Averaging over $m=1,\dots,N$ implies by Lemma~\ref{consequences} that
\begin{align*}
\frac{1}{N} \sum_{m=1}^N Q_1 & = b^2 \hbar^2 \frac{1}{2} \lVert T^{-1} \rVert_{HS}^2 \, |\nabla u|^2 \\
& = \hbar^2 |\nabla u|^2 , \\
\frac{1}{N} \sum_{m=1}^N Q_2 & = b c \hbar \beta \Real \big\{ i \overline{u} \nabla u \big[ T^{-1}MT - (T^{-1}MT)^\dagger \big] x \big\} \\
& = 2\hbar \beta \Real \big\{ i \overline{u} (\nabla u) Mx \big\} , \\
\frac{1}{N} \sum_{m=1}^N Q_3 & = \frac{1}{2} c^2 \beta^2 \lVert MT \rVert_{HS}^2 \, |u|^2 |x|^2 \\
& = \frac{1}{4} \beta^2 |u|^2 |x|^2 ,
\end{align*}
where we used for $Q_2$ that $\tr M = 0$ and where we simplified the three expressions by substituting the definitions of the matrices $M$ and $T$ and the constants $b$ and $c$.

Hence by averaging \eqref{rayleighprinc} over $m=1,\dots,N$ we deduce that 
\begin{align*}
& \sum_{j=1}^n \lambda_j(E,b \hbar,c \beta) \\
& \leq \sum_{j=1}^n \int_D \Big[ \hbar^2 |\nabla u_j|^2 + 2 \hbar \beta \Real \big\{ i \overline{u_j} (\nabla u_j) M x  \big\} + \frac{1}{4} \beta^2 |u_j|^2 |x|^2 \Big] \, dx \\
& = \sum_{j=1}^n \int_D \big| \big( i \hbar \nabla + \beta (Mx)^\dagger \big) u_j  \big|^2 \, dx  \\& = \sum_{j=1}^n \lambda_j(D,\hbar,\beta) 
\end {align*}
which completes the proof of the inequality in the theorem.

\section{\bf Proof of Theorem~\ref{mag}: the case of equality}
\label{equality_proof}

We continue to treat the Dirichlet eigenvalues. The Neumann case proceeds exactly the same way. 

\subsubsection*{Sufficient conditions for equality}

When $T$ is an orthogonal matrix, equality holds in the theorem because the eigenvalues are invariant under orthogonal transformations (see Appendix~\ref{invariance}) and $T^{-1}=T^\dagger$ has Hilbert--Schmidt norm $(\tr T^\dagger T)^{1/2}=(\tr \mathrm{Id})^{1/2}=\sqrt{2}$. Scalar multiples preserve equality too, because 
\[
\lambda_j \big( rD , r\hbar, \beta/r \big) = \lambda_j(D,\hbar,\beta)
\]
by Lemma~\ref{dilation}. 

\subsubsection*{Necessary conditions for equality, when $n=1$} Now assume $\beta \neq 0$ and that equality holds in the theorem for the first eigenvalue, so that
\begin{equation} \label{fundeq}
 \lambda_1 \Big({T(D)}, \frac{\sqrt{2}}{\lVert T^{-1} \rVert_{HS}} \hbar, \frac{\sqrt{2}}{\lVert T^{-1} \rVert_{HS}} \frac{\beta}{|\det T|} \Big) = \lambda_1 \big(D,\hbar,\beta) .
\end{equation}
Our task is to show that $T$ is a scalar multiple of an orthogonal matrix. 

We may assume $T=\left( \begin{smallmatrix} t_1 & 0 \\ 0 & t_2 \end{smallmatrix} \right)$ is diagonal with $t_1 , t_2 > 0$, by applying the singular value decomposition exactly as in the previous section. If $t_1 = t_2$ then the diagonal matrix $T$ is a scalar multiple of the identity, and so the original $T$ was a scalar multiple of an orthogonal matrix, as desired. 

Thus we may suppose from now on that $t_1 \neq t_2$. We will deduce a contradiction, thus ruling out this possibility and completing the proof of the equality case in the theorem. 

Write $u=u_1$ for a first eigenfunction on $D$, with respect to the vector potential $F(x)=\frac{\beta}{2}(-x_2,x_1)$. Then $u$ satisfies the eigenequation $(i\hbar \nabla+F)^2 u=\lambda_1(D,\hbar,\beta) u$, which expands to say
\begin{equation} \label{linear1}
-\hbar^2 (u_{x_1 x_1} + u_{x_2 x_2}) + i \hbar \beta (-x_2,x_1) \cdot \nabla u + \frac{\beta^2}{4} (x_1^2+x_2^2) u = \lambda_1(D,\hbar,\beta) u .
\end{equation}
Inspecting the proof of the theorem in the previous section, we see that one of the trial functions on $T(D)$ is $v=u \circ T^{-1}$, in other words $v(x_1,x_2)=u(x_1/t_1,x_2/t_2)$. (This trial function arises when $m=N$, since $U_N=\mathrm{Id}$.) Since equality must hold in the Rayleigh principle \eqref{rayleighprinc} with $n=1$, we deduce that this trial function must be a first eigenfunction on $T(D)$ using Planck constant $b \hbar$ and field parameter $c\beta$. That is, $v$ satisfies the eigenequation
\[
-b^2 \hbar^2 (v_{x_1 x_1} + v_{x_2 x_2}) + i b c \hbar \beta (-x_2,x_1) \cdot \nabla v + \frac{c^2 \beta^2}{4} (x_1^2+x_2^2) v =  \lambda_1 \big( T(D),b\hbar,c\beta \big) v .
\]
In terms of $u$, and using equality of the eigenvalues in \eqref{fundeq}, we can rewrite this last equation as 
\begin{align} 
- & \frac{2\hbar^2 }{t_1^{-2}+t_2^{-2}} \Big( \frac{u_{x_1 x_1}}{t_1^2} + \frac{u_{x_2 x_2}}{t_2^2} \Big) +\frac{2 i \hbar \beta }{t_1 t_2^{-1} + t_1^{-1} t_2} (-t_2 x_2, t_1 x_1) \cdot \Big( \frac{u_{x_1}}{t_1} , \frac{u_{x_2}}{t_2} \Big) \notag \\ 
+ & \frac{\beta^2}{2(t_1^2+t_2^2)} ( t_1^2 x_1^2 + t_2^2 x_2^2 ) u = \lambda_1(D,\hbar,\beta) u . \label{linear2}
\end{align}

By solving the simultaneous linear equations \eqref{linear1} and \eqref{linear2} (which can be done since $t_1 \neq t_2$) we obtain expressions for the second partial derivatives in terms of lower order derivatives:
\begin{align*} 
-u_{x_1 x_1} & = +\frac{i\beta}{\hbar} x_2 u_{x_1} - \frac{\beta^2}{4\hbar^2} x_2^2 u + \frac{\lambda_1(D,\hbar,\beta)}{2\hbar^2} u , \\
-u_{x_2 x_2} & = - \frac{i\beta}{\hbar} x_1 u_{x_2} - \frac{\beta^2}{4\hbar^2} x_1^2 u + \frac{\lambda_1(D,\hbar,\beta)}{2\hbar^2} u .
\end{align*}
These formulas are equivalent to
\begin{align} 
-(e^{i\beta x_1 x_2/2\hbar} \, u)_{x_1 x_1} & = \omega^2 (e^{i\beta x_1 x_2/2\hbar} \, u) , \label{eigen1new} \\
-(e^{i\beta x_1 x_2/2\hbar} \, \overline{u})_{x_2 x_2} & = \omega^2 (e^{i\beta x_1 x_2/2\hbar} \, \overline{u}) , \label{eigen2new}
\end{align}
where $\omega=\sqrt{\lambda_1(D,\hbar,\beta)/2\hbar^2}$. Note $\omega>0$ because the eigenvalue $\lambda_1$ is positive (Lemma~\ref{poseigen}). 
 

Given any point $y \in D$, take a rectangular neighborhood $R \subset D$ centered at $y$ with the sides of the rectangle being parallel to the axes. On that neighborhood we can solve the ODEs \eqref{eigen1new} and \eqref{eigen2new}, finding that
\begin{align*} 
e^{i\beta x_1 x_2/2\hbar} \, u & = a(x_2)e^{i\omega x_1} + b(x_2)e^{-i\omega x_1}, \\
e^{-i\beta x_1 x_2/2\hbar} \, u & = c(x_1)e^{i\omega x_2} + d(x_1)e^{-i\omega x_2} , 
\end{align*}
for some functions $a,b,c,d$. Hence
\begin{equation} \label{frequencies}
a(x_2)e^{i\omega x_1} + b(x_2)e^{-i\omega x_1} = e^{i\beta x_1 x_2/\hbar} \big( c(x_1)e^{i\omega x_2} + d(x_1)e^{-i\omega x_2} \big) .
\end{equation}
Let $0<z_1<\pi/\omega$ with $4z_1$ being smaller than the width of $R$. Consider equation \eqref{frequencies} with two different $x_1$-values, namely $x_1=y_1+z_1$ and $x_1=y_1+2z_1$, and with any $x_2$-value that cuts through the rectangle $R$. The resulting two equations provide simultaneous linear equations for the unknowns $a(x_2)$ and $b(x_2)$. These simultaneous equations can be solved (since $0<2\omega z_1<2\pi$). The solution for $a(x_2)$ has the form
\begin{align}
a(x_2) & = e^{i\beta y_1 x_2/\hbar} \times \big( \text{linear combination of terms $e^{i\gamma x_2}$ with} \notag \\
& \hspace*{.3in} \text{frequencies $\gamma=\pm \omega + \nu \beta z_1/\hbar$, where $\nu=1,2$} \big) . \label{freq-eq}
\end{align}
The same conclusion holds for any other sufficiently small value of $z_1$, which implies that $a(x_2)$ is identically zero, as we now explain. 

Since $\beta \neq 0$ we may choose $\delta=\omega \hbar / |\beta| k$ where the integer $k \geq 9$ is taken so large that $z_1=4\delta$ is smaller than $\pi/\omega$ and $4z_1$ is less than the width of the rectangle $R$. Let $\e=\beta/|\beta|$ so that $\e=\pm 1$. Then for the frequency $\gamma=\pm \omega + \nu \beta z_1/\hbar$ we have $e^{i\gamma x_2}=e^{i(\pm k + 4\e)(\omega x_2/k)}$ when $\nu=1$, and  $e^{i\gamma x_2}=e^{i(\pm k + 8\e)(\omega x_2/k)}$ when $\nu=2$; these four exponentials determine the linear combination for $a(x_2)$ in \eqref{freq-eq}. On the other hand, we could choose $z_1=\delta$, in which case the linear combination for $a(x_2)$ would involve the four exponentials $e^{i(\pm k + \e)(\omega x_2/k)}$ and $e^{i(\pm k + 2\e)(\omega x_2/k)}$. These two different linear combinations for $a(x_2)$ involve eight distinct frequencies: the eight integers $\pm k + \e,\pm k + 2\e,\pm k + 4\e,\pm k + 8\e$ are distinct because $k \geq 9$. The linear combinations for $a(x_2)$ are equal for a whole interval of $x_2$-values, and so the coefficient of each of the eight exponentials must be $0$, because a nontrivial trigonometric polynomial can have at most finitely many zeros. Hence $a(x_2) \equiv 0$.

Similarly $b(x_2)$ must vanish identically and so $u$ is identically zero on $R$, and hence on all of $D$, which is impossible because $u$ is an eigenfunction. This contradiction concludes the proof.

\section{\bf Proof of Corollary~\ref{2dim}}
\label{2dim_proof}

Assume $\Omega=T(D)$ for some linear transformation $T$. Pulling out the coefficient of the Planck constant in Theorem~\ref{mag} implies that
\[
\frac{2}{\lVert T^{-1} \rVert_{HS}^2} \sum\limits_{j=1}^{n} \lambda_j\Big({T(D)}, \hbar, \frac{\beta}{|\det T|} \Big) \leq \sum\limits_{j=1}^{n} \lambda_j\big(D,\hbar,\beta) .
\]
The Hilbert--Schmidt norm of $T^{-1}$ can be interpreted in terms of moment of inertia of the image domain $T(D)$, with 
\[
\frac{2}{\lVert T^{-1} \rVert_{HS}^2} = \frac{A^3}{I}\big(T(D)\big) \Big/ \frac{A^3}{I}(D) \notag \\
\]
by \cite[Lemma~5.3]{LS11b}. And obviously
\[
\frac{1}{|\det T|} = \frac{A(D)}{A\big( T(D) \big)} .
\]
Substituting these expressions now proves the desired inequality in the Corollary, after replacing $\beta$ with $\beta/A(D)$ on both sides. 

The ``maximality'' statement follows from the ``equality'' statement in Theorem~\ref{mag}.

\section{\bf Proof of Theorem~\ref{magrobin}}
\label{magrobin_proof}

The proof follows the Dirichlet and Neumann cases (see the proof of Theorem~\ref{mag}), except that we must handle also the boundary integral appearing in the Rayleigh quotient \eqref{robinrayleigh}. This we now do, by following Laugesen and Siudeja's proof from the case of vanishing magnetic field \cite[Theorem~3.3]{LS11b}. 

The boundary contribution to the Rayleigh quotient of the trial function $v$ on the domain $E=T(D)$ is
\begin{align*}
\frac{\int_{\partial E} |v|^2 \, ds(x)}{\int_E |v|^2 \, dx}
& = \frac{\int_{\partial E} |u(UT^{-1}x)|^2 \, ds(x)}{\int_E |u(UT^{-1}x)|^2 \, dx} \\
& = \frac{\int_{\partial D} |u(Ux)|^2 |T\tau(x)| \, ds(x)}{\int_D |u(Ux)|^2 \, dx \cdot |\det T|}
\end{align*}
by $x \mapsto Tx$, where $\tau(x)$ denotes the unit tangent vector to $\partial D$ at $x$. Geometrically, $|T\tau(x)|$ is the factor by which $T$ stretches the tangent direction to $\partial D$ at $x$.

The rotational symmetry of $D$ ensures that tangent vectors rotate according to $\tau(U^{-1}x) = U^{-1} \tau(x)$. Thus after replacing $x$ with $U^{-1}x$ in the last formula we obtain 
\[
\frac{\int_{\partial E} |v|^2 \, ds(x)}{\int_E |v|^2 \, dx}
=  |\det T|^{-1} \int_{\partial D} |u(x)|^2 |TU^{-1}\tau(x)| \, ds(x) .
\]
Choosing $U=U_m$ and averaging the preceding quantity over $m$, and then applying Cauchy--Schwarz, gives the upper estimate
\begin{align*}
& |\det T|^{-1} \int_{\partial D} |u(x)|^2 \Big\{ \frac{1}{N} \sum_{m=1}^N  |T U_m^{-1} \tau(x)|^2 \Big\}^{\! \! 1/2} \, ds(x) \\
& = |\det T|^{-1} \, \frac{1}{\sqrt{2}} \lVert T \rVert_{HS} \int_{\partial D} |u(x)|^2 \, ds(x) \notag
\end{align*}
by Lemma~\ref{variant}. After multiplying by $\sigma \sqrt{2}/ \lVert T^{-1} \rVert_{HS}$ and sustituting the relation 
\[
\lVert T^{-1} \rVert_{HS} = \frac{1}{|\det T|} \lVert T \rVert_{HS}
\]
(which is valid for $2 \times 2$ matrices, by the formula for $T^{-1}$ in terms of $T$) we obtain the quantity
\[
\sigma \int_{\partial D} |u(x)|^2 \, ds(x) .
\]
Now we can easily modify the proof of Theorem~\ref{mag} to handle the Rayleigh quotient \eqref{robinrayleigh} for the Robin--de Gennes eigenvalues. 

The equality statement follows exactly as for Theorem~\ref{mag}.

\section*{Acknowledgments}
The authors acknowledge support from the National Science Foundation grant
DMS 08-38434 ``EMSW21-MCTP: Research Experience for Graduate Students''.

We thank Bart{\l}omiej Siudeja for creating Figure~\ref{fig:lineartrans}, and Rupert Frank for alleviating our ignorance on certain issues.

\appendix

\section{\bf Invariance and positivity of the spectrum}
\label{invariance}

Our proofs rely on rotations, reflections and dilations of the domain. These reductions are justified by the invariance lemmas in this appendix, which we prove for the sake of readers unfamiliar with the magnetic Laplacian.

Throughout this appendix, $\Omega$ is a bounded plane domain, further assumed to have Lipschitz boundary in the Neumann and Robin--de Gennes situations. We give proofs for the Dirichlet eigenvalues only. The proofs are identical for Neumann conditions, except using $H^1$ instead of $H^1_0$. For Robin--de Gennes conditions, we indicate (when needed) how to treat the boundary integral in the Rayleigh quotient. 

In the first lemma, we go against the notation used in the rest of the paper and write the eigenvalues as depending on the vector potential rather than on its curl, the magnetic field. Let $\widetilde{F},F \in C^\infty(\R^2;\R^2)$ be vector potentials defined on the whole plane that generate magnetic fields $\widetilde{B}=\nabla \times \widetilde{F},B=\nabla \times F$, respectively. These fields are not necessarily constant. 

\begin{lemma}[Gauge invariance] \label{gaugelemma}
If two vector potentials differ by a gradient vector, then they generate the same magnetic eigenvalues. That is, if $\widetilde{F}=F+\nabla f$ for some function $f \in C^\infty(\R^2)$ then
\[
\lambda_j(\Omega,\hbar,\widetilde{F}) = \lambda_j(\Omega,\hbar,F) , \qquad j =1,2,3,\dots
\]
and similarly for the Robin--de Gennes and Neumann eigenvalues.
\end{lemma}

\begin{proof}
For the Dirichlet spectrum, given a trial function $u \in H^1_0(\Omega;\C)$ we define $v=e^{if}u \in H^1_0(\Omega;\C)$. Then $|v|^2=|u|^2$ and 
\[
(i\hbar\nabla + \widetilde{F})v = e^{if} (i\hbar\nabla + F)u
\]
by direct calculation. Taking the magnitude and then squaring and integrating shows that the Rayleigh quotient of $v$ with potential $\widetilde{F}$ equals the Rayleigh quotient of $u$ with potential $F$. Hence the lemma follows from the variational characterization of the eigenvalues  \cite[p.\,97]{B80}. 
\end{proof}

Next we show that the eigenvalues depend only on the curl of the potential, so that we are justified in writing $\lambda_j(\Omega,\hbar,B)$ instead of $\lambda_j(\Omega,\hbar,F)$.

\begin{lemma}[Independence from choice of potential] \label{indeppot}
If $\widetilde{B}=B$ on $\R^2$ then $\tilde{F}$ and $F$ generate the same magnetic eigenvalues.
\end{lemma}

\begin{proof}
Since $\R^2$ is simply connected and the vector potentials $\widetilde{F}$ and $F$ have the same curl, the potentials differ by a gradient vector. Hence they generate the same magnetic spectrum, by Lemma~\ref{gaugelemma}.
\end{proof}

\begin{lemma}[Independence from direction of magnetic field] \label{direction}
The fields $B$ and $-B$ generate the same magnetic eigenvalues.
\end{lemma}

\begin{proof}
Given a trial function $u$ to be used for the field $B$, define a trial function $v=\overline{u}$ for the field $-B$. Obviously $|u|=|v|$ and $|(i\hbar \nabla + F)u|=|(i\hbar \nabla - F)v|$, and so the respective Rayleigh quotients of $u$ and $v$ are equal, which proves the lemma by the variational characterization of eigenvalues. 
\end{proof}

For the remainder of the appendix we assume the magnetic field is constant, with $B=(0,0,\beta)$ for some $\beta \in \R$.

\begin{lemma}[Invariance under rotation] \label{rotation}
If $V$ is a $2 \times 2$ rotation matrix then $\lambda_j(V(\Omega),\hbar,\beta)=\lambda_j(\Omega,\hbar,\beta)$ for each $j$, and similarly for the Robin--de Gennes and Neumann eigenvalues.
\end{lemma}

\begin{proof}
Choose the vector potential $F(x)=\beta(Mx)^\dagger=\frac{\beta}{2}(-x_2,x_1)$ where $M=
\left(
 \begin{smallmatrix}
  0 & -1/2 \\
 1/2 & 0
 \end{smallmatrix}
\right)$, which gives the field $\nabla \times F = (0,0,\beta)$. We use this vector potential on both the domains $\Omega $ and $V(\Omega)$. 

Given a trial function $u$ on $\Omega $, define $v=u \circ V^{-1}$  on $V(\Omega)$. Obviously $u$ and $v$ have the same $L^2$ norms, and thus the denominators of their respective Rayleigh quotients are equal. The numerator of the Rayleigh quotient for $v$ has integrand 
\begin{align*}
|(i\hbar\nabla + F)v(x)|^2 
& = \big| i \hbar(\nabla  u)(V^{-1}x)V^{-1}+u(V^{-1}x)F(x) \big|^2 \\
& = |i \hbar(\nabla  u)(V^{-1}x)+u(V^{-1}x)F(V^{-1}x)|^2
\end{align*}
because $F(x)=F(V^{-1}x)V^{-1}$ (using that rotations commute in $2$ dimensions and so $M^\dagger$ commutes with $V$). Hence after integrating and changing variable with $x \mapsto Vx$ we see
that the numerators of the Rayleigh quotients of $v$ and $u$ are equal. (Also, in the Robin case $\int_{\partial V(\Omega)} |v|^2 \, ds = \int_{\partial \Omega} |u|^2 \, ds$.)

The lemma follows immediately from the variational characterization of the eigenvalues.  
\end{proof}

\begin{lemma}[Invariance under reflection]
 \label{reflection}
If $V$ is a $2 \times 2$ reflection matrix then $\lambda_j(V(\Omega),\hbar,\beta)=\lambda_j(\Omega,\hbar,\beta)$ for each $j$, and similarly for the Robin--de Gennes and Neumann eigenvalues.
\end{lemma}

\begin{proof}
In view of the rotation invariance in the preceding lemma, we may assume the reflection occurs across the $x_2$-axis, with $V=\left(\begin{smallmatrix} -1 & 0 \\ 0 & 1 \end{smallmatrix}\right)$. 
Choose vector potential $F(x)=\beta(0,x_1)$, for which $\nabla \times F = (0,0,\beta)$, and check that $F(x)=-F(V^{-1}x)V^{-1}$. Now adapt the proof of Lemma~\ref{rotation} to show  $\lambda_j(V(\Omega),\hbar,\beta)=\lambda_j(\Omega,\hbar,-\beta)$. Then replace $-\beta$ with $\beta$ on the right side, by Lemma~\ref{direction}.
\end{proof}

\begin{lemma}[Invariance under translation]
 \label{translation}
If $y \in \R^2$ then $\lambda_j(\Omega +y,\hbar,\beta)=\lambda_j(\Omega,\hbar,\beta)$ for each $j$, and similarly for the Robin--de Gennes and Neumann eigenvalues.
\end{lemma}

\begin{proof}
Straightforward, by translating the potential along with the domain.
\end{proof}

\begin{lemma}[Invariance under dilation]
 \label{dilation}
If $r>0$ then $\lambda_j(r\Omega,r\hbar,\beta/r)=\lambda_j(\Omega,\hbar,\beta)$ for each $j$, and similarly for the Neumann eigenvalues. For the Robin--de Gennes eigenvalues, $\rho_j(r\Omega,r\hbar,\beta/r,r\sigma)=\rho_j(\Omega,\hbar,\beta,\sigma)$.
\end{lemma}

\begin{proof}
We use the potential $F(x)=\beta(0,x_1)$ on $\Omega$, giving field strength $\beta$, and potential $F(x/r)= \frac{\beta}{r} (0,x_1)$ on $r \Omega $, giving field strength $\beta/r$. 

Given a trial function $u$ on $\Omega $, define $v(x)=r^{-1}u(x/r)$  on $r \Omega $. Then $u$ and $v$ have the same $L^2$ norms, and so the denominators of their Rayleigh quotients are the same. The numerator of the Rayleigh quotient for $v$ is 
\[
\int_{r \Omega} |(ir\hbar\nabla + F(x/r))v(x)|^2 \, dx 
= \int_{\Omega} \big| i \hbar (\nabla  u)(x)+F(x) u(x) \big|^2 \, dx ,
\]
which equals the numerator of the Rayleigh quotient for $u$. In the Robin case, note also that the boundary term transforms according to $r\sigma \int_{\partial (r \Omega)} |v|^2 \, ds = \sigma\int_{\partial \Omega} |u|^2 \, ds$.

Once again the lemma follows from the variational characterization of the eigenvalues.  
\end{proof}

Finally we show the eigenvalues are positive, when the field is nonzero. 
\begin{lemma}[Positivity of the energy] \label{poseigen}
If $\beta \neq 0$ then the first eigenvalues $\lambda_1(\Omega,\hbar,\beta), \rho_1(\Omega,\hbar,\beta)$ and $\mu_1(\Omega,\hbar,\beta)$ are all positive. 
\end{lemma}
\begin{proof}
Suppose one of the first eigenvalues equals $0$, and let $u$ be a corresponding eigenfunction using the potential $F(x)=(0,\beta x_1)$. Since the Rayleigh quotient of $u$ equals $0$, we deduce $i \hbar \nabla u + F u \equiv 0$ in $\Omega$. 

Consider a point at which $u \neq 0$. Near that point we have $i\hbar \nabla \log u \equiv -F$, so that $F$ is locally a gradient vector. Therefore $F$ must satisfy the compatibility condition $\partial_1 F_2 = \partial_2 F_1$, which says $\beta=0$.\end{proof}
The first Dirichlet and Robin eigenvalues remain positive even when $\beta=0$. The Neumann eigenvalue $\mu_1$ equals $0$, when $\beta=0$. 

Incidentally, monotonicity of the first Neumann eigenvalue $\mu_1(\Omega,\hbar,\beta)$ with respect to the field strength is relevant in superconductivity. Monotonicity has been proved for large $\beta$ by Fournais and Helffer \cite{FH07}.

\end{document}